\documentclass[]{interact}

\usepackage{epstopdf}% To incorporate .eps illustrations using PDFLaTeX, etc.
\usepackage[caption=false]{subfig}% Support for small, `sub' figures and tables
\usepackage[numbers,sort&compress]{natbib}% Citation support using natbib.sty
\bibpunct[, ]{[}{]}{,}{n}{,}{,}% Citation support using natbib.sty
\makeatletter% @ becomes a letter
\def\NAT@def@citea{\def\@citea{\NAT@separator}}% Suppress spaces between citations using natbib.sty
\makeatother% @ becomes a symbol again

\theoremstyle{plain}% Theorem-like structures provided by amsthm.sty
\newtheorem{theorem}{Theorem}[section]
\newtheorem{lemma}[theorem]{Lemma}
\newtheorem{corollary}[theorem]{Corollary}
\newtheorem{proposition}[theorem]{Proposition}
\newtheorem{definition}[theorem]{Definition}
\newtheorem{example}[theorem]{Example}
\newtheorem{remark}[theorem]{Remark}

\usepackage{comment}

\newcommand{\R}{{\mathcal{R}}}

\usepackage{amsmath,amsthm,amsfonts,amssymb,enumerate}
\usepackage{paralist}
\usepackage{tabto}
\usepackage{textcomp}
\usepackage[utf8]{inputenc}
\usepackage{tikz,xcolor,hyperref}

\definecolor{lime}{HTML}{A6CE39}
\DeclareRobustCommand{\orcidicon}{%
	\begin{tikzpicture}
	\draw[lime, fill=lime] (0,0)
	circle [radius=0.16]
	node[white] {{\fontfamily{qag}\selectfont \tiny ID}};
	\draw[white, fill=white] (-0.0625,0.095)
	circle [radius=0.007];
	\end{tikzpicture}
	\hspace{-2mm}
}

\foreach \x in {A, ..., Z}{%
	\expandafter\xdef\csname orcid\x\endcsname{\noexpand\href{https://orcid.org/\csname orcidauthor\x\endcsname}{\noexpand\orcidicon}}
}

\begin{document}

%\articletype{Original Article}

{\title{\large 
Characterizations  of weighted (b,c) inverse %in a ring
}}
\author{
\name{Bibekananda Sitha\textsuperscript{a}, Jajati Keshari Sahoo\textsuperscript{a}\orcidA\thanks{CONTACT Jajati Keshari Sahoo Email: jksahoo@goa.bits-pilani.ac.in}, Ratikanta Behera\textsuperscript{b}\orcidB}
\affil{\textsuperscript{a}Department of Mathematics,  BITS Pilani, K K Birla Goa Campus, Zuarinagar, Goa, India.
\textsuperscript{b} Department of Mathematics, University of Central Florida, Orlando, 32826, FL, USA.
}
}
\maketitle

\begin{abstract} 
The notion of weighted $(b,c)$-inverse of an element in rings were  introduced, very recently [Comm. Algebra, 48 (4) (2020): 1423-1438]. In this paper, we further elaborate on this theory by establishing a few characterizations  of this inverse and  their  relationships  with other $(v, w)$-weighted $(b,c)$-inverses. We introduce some necessary and sufficient conditions for the existence of the hybrid $(v, w)$-weighted $(b,c)$-inverse and  annihilator  $(v, w)$-weighted $(b,c)$-inverse of elements in rings. In addition to this, we explore a few sufficient conditions for the reverse-order law of the annihilator  $(v, w)$-weighted $(b,c)$-inverses. 
\end{abstract}



\section{Introduction}

\subsection{Background and motivation}

Let $\R$ be a ring with involution. The theory of generalized inverses has generated a tremendous amount of interest in many research areas in mathematics \cite{rakic,BenIsrael03,SahBe20,Grev1966}. Several types of generalized inverses are available in the literature, such as Moore-Penrose inverse\cite{koliha}, group inverse \cite{Hartwig1977},  Drazin inverse \cite{Drazin58}, core inverse\cite{rakic}. It is worth to mention, Drazin in \cite{Drazin2012} introduced $(b,c)$ inverse in the setting of a semigroup, which is a generalization of Moore–Penrose inverse. The concept of annihilator $(b, c)$- and hybrid $(b, c)$-inverse were established as generalizations of $(b, c)$-inverses in \cite{Drazin2012}. Further, several characterizations hybrid and annihilator $(b,c)$-inverse have been discussed in \cite{Zhu2018}.  Mary proposed the inverse along an element (see \cite{Mary11} Definition 4), as a new type of generalized inverse. Many researchers \cite{weakg,Drazin13,Drazin13left} were explored numerous properties of these inverses and interconnections with other generalized inverses. Among the extensive work of generalized inverses, there has been a growing interest with ``weighted'' generalized inverses \cite{mosic18,das,PredragKM17} in recent years for encompassing the above-mentioned generalized inverses.

In connection with the theory $(b, c)$-inverses (see \cite{Drazin2012}, Definition 1.3) and Bott-Duffin inverse \cite{BottDu53}, Drazin explored the Bott-Duffin $(e, f)$-inverse  (see \cite{Drazin2012}, Definition 3.2) in a semigroup. Further, very large class of ``$(v, w)$-weighted version'' of $(b, c)$-inverses are introduced in \cite{Drazin20Weighted} e.g., annihilator $(v,w)$-weighted $(b,c)$-inverses (see Definition 4.1) and and hybrid $(v,w)$-weighted $(b,c)$-inverses (see Definition 4.2).  The vast work on hybrid and annihilator $(b,c)$-inverse along with the above weighted $(b, c)$-inverse, motivate us to study a few characterizations and representations for hybrid and annihilator $(v, w)$-weighted $(b, c)$-inverse. 

More precisely, the main contributions of this paper are as follows:

\begin{enumerate}
    \item[\textnormal{$\bullet$}] A few necessary and sufficient conditions for the existence of the $(v, w)$-weighted $(b,c)$ inverses of elements in rings are introduced. 
    \item[\textnormal{$\bullet$}] Some characterizations of the  $(v,w)$-weighted hybrid $(b,c)$-inverse and annihilator $(v, w)$-weighted $(b, c)$-inverses are investigated.
   \item[\textnormal{$\bullet$}] The construction of $(v,w)$-weighted hybrid $(b,c)$-inverse via group inverse is presented. 
\end{enumerate}

\subsection{Outlines}
Our presentation is organized as follows. We present some notations and existing definitions in Section 2. In Section 3, we have discussed a few characterizations for the $(v,w)$-weighted $(b,c)$-inverse. Various equivalent characterizations of the hybrid $(v,w)$-weighted $(b,c)$-inverse are presented in Section 4. In Section 5, we study the representation of the annihilator $(v,w)$-weighted $(b,c)$-inverse.  The contribution of our work is summarized in Section 6.

\section{Preliminaries}
Throughout this paper, we consider $\R$ be an associative ring with unity 1. The sets of all left annihilators and right annihilators of $a$ is respectively defined by 
\begin{center}
 $lann(a)=\{x\in \R: xa=0\}$ and $rann(a)=\{z\in \R: az=0\}$. 
\end{center}

We denote the left and right ideals by $a\R = \{ar~:~r\in \R\}$ and $\R a = \{za~:~ z\in\R\}$. An element $y\in \R$ is called generalized or inner inverse of $a\in \R$ if $aya=y$. If such $y$ exist, we call $a$ is regular. The set of inner inverse of $a$ is denoted by $a\{1\}$ and the  elements are represented by $a^{-}$.  
The following result gives the relation between ideals and annihilators.

\begin{proposition}\label{prop2.1}\cite{Koliha03}
If $a$ is idempotent, then $rann(a)=(1-a)\R$ and $lann(a)=\R (1-a)$.
\end{proposition}

Next we recall the definition of group inverse \cite{Drazin58} of an element in $\R$. An element $y$ is called group inverse of $a\in \R$ if $aya=a,~yay=y$, and $ay=ya$. The group inverse of $a$ is denoted by $a^{\#}$. The necessary and sufficient condition for the existence of group inverse is stated in the next result.

\begin{lemma}\label{LemmaHar}[Theorem 1, \cite{HartJia77}]
Let $a\in \R$. Then $a$ is group invertible if and only if $a\in a^2 \R\cap \R a^2.$ 
\end{lemma}

We next recall the ``$(v, w)$-weighted'' version of $(b,c)$ inverse from \cite{Drazin20Weighted}.

\begin{definition}\label{wbc} [Theorem 2.1 (i), \cite{Drazin20Weighted}]
Let $a,b,c,v,w\in \R$. An element $y\in \R$ satisfying 
\begin{center}
   $y\in b\R wy\cap yv \R c$, $yvawb=b$ and $cvawy=c$, 
\end{center}
is called the $(v,w)$-weighted $(b,c)$-inverse of $a$ and denoted by $a^{v,w}_{b,c}$.
\end{definition}

Drazin in Proposition 3.2, \cite{Drazin20Weighted} proved that $a^{v,w}_{b,c}$ is unique if exists.

%Further, proved the following equivalent characterization for the $(v,w)$-weighted $(b,c)$-inverse. 
%\begin{lemma}[Theorem 2.1 (ii) and (iii)), \cite{Drazin20Weighted}]
%Let $a,b,c,v,w,y\in \R$. Then the following statements %are equivalent. 
%\begin{enumerate}[(i)]
%    \item $yvawy=y$, $yv \R=b \R$ and $\R wy=\R c$.
%    \item $yvawy=y$, $y\in b \R\cap \R c$, $b\in yv \R$ %and $c\in \R wy$.
%\end{enumerate}
%\end{lemma}

An equivalent characterization of the $(v,w)$-weighted $(b,c)$-inverse is presented in the next result.
 
 \begin{lemma}(Theorem 2.5 and Theorem 2.1, \cite{Drazin20Weighted})\label{lem3.6}  Let $a,b,c,v,w\in \R$.Then the following conditions are equivalent:
\begin{enumerate}[\rm(i)]
\item $a$ has a $(v,w)$-weighted $(b,c)$-inverse.
\item $c\in cvawb\R$ and $b\in \R cvawb$.
\item there exists $y\in \R$ such that $yvawy=y$, $yv\R=b\R$ and $\R wy=\R c$.
\end{enumerate}
\end{lemma}

 Following the definition (see \cite{Mary11} Definition 4) of the inverse along an element of $\R$, we now define  $(v,w)$-weighted inverse of $a$ along $d\in \R$ below.

\begin{definition}\label{walongd}
Let $a,d,v,w\in \R$. An element $y\in \R$ satisfying 
\begin{center}
  $yvawd=d=dvawy$, $\R wy\subseteq Rd$, and $yv\R\subseteq d\R$ \end{center}
is called the $(v,w)$-weighted inverse of $a$ along $d\in \R$ and denoted by $a^{v,w}_{\parallel d}$.
 \end{definition}

\begin{example}\rm
Let $\R=M_2(\mathbb{R})$, with $a = \begin{bmatrix}
1 & 1\\
0 & 0
\end{bmatrix}$, $v = \begin{bmatrix}
1 & 1\\
0 & -1
\end{bmatrix}$, $w = \begin{bmatrix}
0 & 1\\
1 & 0
\end{bmatrix}$, and $d = \begin{bmatrix}
1 & 2\\
0 & 0
\end{bmatrix}$. Since the matrix $y= \begin{bmatrix}
1 & 2\\
0 & 0
\end{bmatrix}$, satisfies 
\begin{center}
   $yvawd = \begin{bmatrix}
1 & 2\\
0 & 0
\end{bmatrix}\cdot\begin{bmatrix}
1 & 1\\
0 & -1
\end{bmatrix}\cdot\begin{bmatrix}
1 & 1\\
0 & 0
\end{bmatrix}\cdot\begin{bmatrix}
0 & 1\\
1 & 0
\end{bmatrix}\cdot\begin{bmatrix}
1 & 2\\
0 & 0
\end{bmatrix}\cdot\begin{bmatrix}
1 & 2\\
0 & 0
\end{bmatrix}= d$,
\end{center}
\begin{center}
 $dvawy=\begin{bmatrix}
1 & 2\\
0 & 0
\end{bmatrix}\cdot\begin{bmatrix}
1 & 1\\
0 & -1
\end{bmatrix}\cdot\begin{bmatrix}
1 & 1\\
0 & 0
\end{bmatrix}\cdot\begin{bmatrix}
0 & 1\\
1 & 0
\end{bmatrix}\cdot\begin{bmatrix}
1 & 2\\
0 & 0
\end{bmatrix}=\begin{bmatrix}
1 & 2\\
0 & 0
\end{bmatrix}=d$,   
\end{center}
$wy=\begin{bmatrix}
0 & 0\\
1 & 2
\end{bmatrix}=\begin{bmatrix}
0 & 0\\
1 & 0
\end{bmatrix}\cdot \begin{bmatrix}
1 & 2\\
0 & 0
\end{bmatrix}=r_1d$ and  $yv=\begin{bmatrix}
1 & -1\\
0 & 0
\end{bmatrix} =\begin{bmatrix}
1 & 2\\
0 & 0
\end{bmatrix}\cdot\begin{bmatrix}
1 & -1\\
0 & 0
\end{bmatrix}=dr_{2}$  
 for some $r_{1}=\begin{bmatrix}
0 & 0\\
1 & 0
\end{bmatrix}$ and $r_{2}=\begin{bmatrix}
1 & -1\\
0 & 0
\end{bmatrix}$, it follows that $a^{v,w}_{\parallel d}=y$. 
\end{example}

In view of right [resp. left] hybrid $(v, w)$-weighted $(b, c)$-inverse (see \cite{Drazin20Weighted}, Definition 4.2) and annihilator $(v, w)$-weighted $(b, c)$-inverse (see \cite{Drazin20Weighted}, Definition 4.1) of $a \in \R$, we next present the definition of the hybrid $(v,w)$-weighted $(b,c)$-inverse and  annihilator $(v,w)$-weighted $(b,c)$-inverse  of $a \in \R$. 
\begin{definition}\label{defhyb}
 Let $a,b,c,v,w\in \R$. An element $y\in \R$ satisfying 
 \begin{center}
     $yvawy=y$,  $yv\R=b\R$, and $rann(c)=rann(wy)$,
 \end{center}
 is called the right hybrid (or hybrid) $(v,w)$-weighted $(b,c)$-inverse of $a$ and denoted by $a^{h,v,w}_{b,c}$.
 \end{definition}
In section 4, we will discuss some results for right hybrid inverse $(v,w)$-weighted $(b,c)$-inverse, which can be similarly proved for left hybrid $(v,w)$-weighted $(b,c)$-inverse. So from here onward we call the right right hybrid  $(v,w)$-weighted $(b,c)$-inverse as hybrid  $(v,w)$-weighted $(b,c)$-inverse. 
 
 The existence of hybrid $(v,w)$-weighted $(b,c)$-inverse over a semigroup which proved in \cite{Drazin20Weighted} restated for the ring $\R$, as follows.

\begin{lemma} {(Theorem 4.7, \cite{Drazin20Weighted})}\label{Th4.7Drazin}
Let $1\in \R$ and $a, b, c, v, w  \in \R$. Then $a^{h, v,w}_{b,c}$ exists if and only if $rann(cvawb) \subseteq rann(b)$, and $c \in cvawb \R$. 
\end{lemma}
 
 \begin{definition}\label{defann}
 Let $a,b,c,v,w\in \R$. An element $y\in \R$ satisfying 
 \begin{center}
     $yvawy=y$, $lann(yv)=lann(b)$, and $rann(c)=rann(wy)$, 
 \end{center}
 is called the annihilator $(v,w)$-weighted $(b,c)$-inverse of $a$ and denoted by $a^{a,v,w}_{b,c}$.
 \end{definition}
The uniqueness of $a^{h,v,w}_{b,c}$ and $a^{a,v,w}_{b,c}$ proved by the author of \cite{Drazin20Weighted}.  In the light of the Bott-Duffin inverse \cite{Drazin2012}, we introduce the $(v,w)$-weighted Bott-Duffin inverse as follows.

\begin{definition}\label{botdef}
Let $a,v,w,e,f\in \R$ with $e^2=e$ and $f^2=f$. An element $z\in \R$ is called $(v,w)$-weighted Bott-Duffin $(e,f)$-inverse of $a$ if it satisfies 
\begin{center}
$z=ewz=zvf, ~zvawe=e,~fvawz=f$.    
\end{center}
\end{definition}

 The $(v,w)$-weighted Bott-Duffin $(e,f)$-inverse of the element $a$ is denoted as $a^{b,v,w}_{e,f}$.

\begin{example}\rm
Let $\R=M_{2}(\mathbb{R})$ with 
 $a=\begin{bmatrix}
0 & 0\\
0 & 1
\end{bmatrix},~v=\begin{bmatrix}
0 & -1\\
1 & 0
\end{bmatrix}$, $w=\begin{bmatrix}
1 & 0\\
0 & 1
\end{bmatrix}$, $e=\begin{bmatrix}
0 & 0\\
1 & 1
\end{bmatrix}$, and $f=\begin{bmatrix}
1 & 1\\
0 & 0
\end{bmatrix}$. We can verify that the matrix  $z=\begin{bmatrix}
0 & 0\\
-1 & -1
\end{bmatrix}$ satisfies 
\begin{center}
 $zvawe=\begin{bmatrix}
0 & 0\\
-1 & -1
\end{bmatrix}\cdot\begin{bmatrix}
0 & -1\\
1 & 0
\end{bmatrix}\cdot\begin{bmatrix}
0 & 0\\
0 & 1
\end{bmatrix}\cdot\begin{bmatrix}
1 & 0\\
0 & 1
\end{bmatrix}\cdot\begin{bmatrix}
0 & 0\\
1 & 1
\end{bmatrix}=\begin{bmatrix}
0 & 0\\
1 & 1
\end{bmatrix}=e$,   
\end{center}
$fvawz=f$, and $ewz=z=zvf$. Hence $a^{b,v,w}_{e,f}=z$
\end{example}

\section{Further results on $(v,w)$-weighted $(b,c)$-inverses}

In this section, we derive a few useful representations and properties of  $(v,w)$-weighted $(b,c)$-inverse. 
\begin{proposition}\label{prop14}
Let $v,w,d\in \R$.
Then the following are holds:  
\begin{enumerate}[\rm(i)]
    \item If $\R wy=\R d$ ($\R wy\subseteq\R d$) then $rann(wy)=rann(d)$ ($rann(d)\subseteq rann(wy)$).
    \item If $yv \R=d \R$ ($yv \R\subseteq d\R$) then $lann(yv)=lann(d)$ ($lann(d)\subseteq lann(yv)$).
    \item If $rann(d)\subseteq rann(wy)$, and $d^{-}$ exists, then  $\R wy\subseteq \R d$.
    \item If  $lann(d)\subseteq lann(yv)$, and $d^{-}$ exists, then $yvR\subseteq dR$.
    \end{enumerate}
\end{proposition}

\begin{proof}
(i) Let $\R wy=\R d$. Then $wy=sd$ for some $s\in \R$. Now if  $z\in rann(d)$, then $dz=0$ and $wyz=sdz=0$. Thus $rann(d)\subseteq rann(wy)$. 

Conversely, let $x\in rann(wy)$. Then $wyx=0$. From $\R wy=\R d$, we can write $d=twy$ for some $t\in \R$. Now $dx=twyx=0$. Hence $rann(wy)\subseteq rann(d)$.\\
(ii) We can show by using the similar lines of part (i).\\
(iii) Let $x\in d\{1\}$. Then $(1-xd)\in rann(d)\subseteq rann(wy)$. Which implies $wy=(wyx)d$. Therefore, $\R wy\subseteq \R d$.\\
(iv) Similar to part (iii).
\end{proof}

 \begin{proposition}\label{pro3.7}
 Let $a,b,c,v,w\in \R$. If $a$ has  $(v,w)$-weighted $(b,c)$-inverse, then both $b$ and $c$ are regular.
 \end{proposition}
  \begin{proof}
Let $y$ be the $(v,w)$-weighted $(b,c)$-inverse of $a$. Then by Definition \ref{wbc}, $yvawb=b$, $cvawy=c$ and $y\in bRwy\cap yvRc$. From the ideals, we further obtain  $y=bswy$ and $y=yvtc$ for some $s,t\in \R$. Now $b=yvawb=bswyvawb=bswb$. Thus $b$ is regular. Similarly, we have $c=cvawy=cvawyvtc=cvtc$ which proves the regularity of $c$.
 \end{proof}

We next present a few an equivalent representations  of the $(v,w)$-weighted $(b,c)$-inverse.

\begin{theorem}\label{thm3.9}
Let $a,b,c,v,w\in \R$. Then the following statements are equivalent:
\begin{enumerate}[\rm(i)]
\item $a$ has a $(v,w)$-weighted $(b,c)$-inverse.
\item $b$ is regular, $\R=\R cvaw\oplus lann(b)$, and $lann(vaw)\cap \R c=0$.
\item $\R=\R cvaw\oplus lann(b)$, $lann(vaw)\cap \R c=0$ and $cvawb$ is regular.
\item $c$ is regular, $\R=vawb \R\oplus rann(c)$, and  $rann(vaw)\cap b \R=0$.
\item $\R=vawb \R\oplus rann(c)$, $rann(vaw)\cap b \R=0$ and $cvawb$ is regular.
\end{enumerate}
\end{theorem}

\begin{proof}
(i)$\Rightarrow$(ii) Assume that $a$ has a $(v,w)$-weighted $(b,c)$-inverse. By Proposition \ref{pro3.7}, we have $b$ is regular. From Lemma \ref{lem3.6}, there exist $p,q\in R$ such that $b=pcvawb$ and $c=cvawbq$. Let $r=1-pcvaw$. Then, $r\in lann(b)$. For any $t\in \R$, 
\begin{center}
$t=t.1=t(pcvaw+r)=tpcvaw+tr\in \R cvaw+lann(b)$    
\end{center}
Therefore, $\R=\R cvaw+lann(b)$.  Moreover, if $u\in \R cvaw\cap lann(b)$, then there is $x\in \R$ such that $u=xcvaw$ and $ub=0$. Now $xc=x(cvawbq)=(ub)q=0$,  and  $u=xcvaw=0$. Thus $\R cvaw\cap lann(b)=\{0\}$.

If $m\in lann(vaw)\cap \R c$, then $mvaw=0$ and $m=sc$, for some $s\in \R$. Therefore, $m=sc=scvawbq=mvawbq=0$. Thus $lann(vaw)\cap Rc=\{0\}$.

(ii)$\Rightarrow$(iii) Since $\R= \R cvaw\oplus lann(b)$. Then, $1=gcvaw+h$, where $g\in \R$ and $h\in lann(b)$. Therefore, $b=gcvawb\in \R cvawb$. Which implies $\R b\subseteq \R cvawb$. Since $\R cvawb\subseteq \R b$ is trivial, it follows that   $\R b=\R cvawb$.  From $\R b=\R cvawb$, we have  $b=scvawb$ and $cvawb=tb$ for some  $s,t\in \R$. Now 
\begin{center}
 $cvawb=tb=tbb^{-}b=cvawbb^{-}scvawb$, where  $b^{-}\in b\{1\}$.  
\end{center}
Hence $cvawb$ is regular.

(iii)$\Rightarrow$(i) Let $\R=\R cvaw\oplus lann(b)$. Then $1=gcvaw+h$ for some $g\in \R$ and $h\in lann(b)$. Therefore, $b=gcvawb\in \R cvawb$. Now , we will prove $lann(c)=lann(cvawb)$. Obviously, $lann(c)\subseteq lann(cvawb)$. For $x\in lann(cvawb)$, we have $xcvaw\in lann(b)\cap \R cvaw=\{0\}$, i.e. $xcvaw=0$. This implies that $xc\in lann(vaw)\cap Rc=\{0\}$. Thus $x\in lann(c)$, i.e. $lann(c)=lann(cvawb)$. Now, let $t\in (cvawb)\{1\}$. Since $(1-cvawbt)cvawb=0$, we  have $1-cvawbt\in lann(cvawb)=lann(c)$. Thus, $c=cvawbtc\in cvawb \R$. Thus by Lemma \ref{lem3.6}, $a$ has a $(v,w)$-weighted $(b,c)$-inverse.

(i)$\Rightarrow$(iv)$\Rightarrow$(v)$\Rightarrow$(i)  similar to (i)$\Rightarrow$(ii)$\Rightarrow$(iii)$\Rightarrow$(i).
\end{proof}

\begin{theorem}\label{thm3.4}
Let $a,b,c,w,v\in \R$. Then the following statements are equivalent:
\begin{enumerate}[\rm(i)]
\item $y=a^{v,w}_{b,c}$.
\item $yvawy=y$, $yv\R=b\R$ and $\R wy=\R c$.
\item $yvawy=y$, $lann(yv)=lann(b)$, $\R wy=\R c$, and $b$ is regular.
\item $yvawy=y$, $yv\R=b\R$, $rann(wy)=rann(c)$, and $c$ is regular.
\item $yvawy=y$, $lann(yv)=lann(b)$ and $rann(wy)=rann(c)$, and both $b,c$ are regular.
\item both $b,c$ are regular, $y=bb^{-}y$, $bb^{-}=yvawbb^{-}$, $yc^{-}c=y$, and $c^{-}c=c^{-}cvawy$.
\item both $b,c$ are regular, $bb^{-}\in \R(c^{-}cvawbb^{-})$, and $c^{-}c\in (c^{-}cvawbb^{-})\R$.
\item both $b,c$ are regular, and there exists $s,t\in \R$ such that $bb^{-}=tc^{-}cvawbb^{-}$, $c^{-}c=c^{-}cvawbb^{-}s$.
\end{enumerate}
\end{theorem}

\begin{proof}
(i)$\Leftrightarrow$(ii). By Lemma \ref{lem3.6}.\\
(ii)$\Rightarrow$(iii) Let $yv\R = b\R$ and $b\R = yv\R$. Then $yv=bs$  and $b=yvt$ for some $s,t\in \R$.  The regularity of $b$ follows from the below expression:
\begin{center}
  $b=yvt=(yvawy)vt=yvaw(yvt)=(yv)awb=b(saw)b$.  
\end{center}
For any $z\in lann(yv)$, we have $zyv=0$. Now $zb=zyvt=0$. Thus $z\in lann(b)$ and subsequently $lann(yv)\subseteq lann(b)$. The reverse inclusion $lann(b)\subseteq lann(yv)$, we can show similarly. Therefore, $lann(yv)=lann(b)$.\\
(iii)$\Rightarrow$(iv) Let $\R wy=\R c$. Then $wy=sc$ and $c=twy$ for some $s,t\in \R$. Now
\begin{center}
    $c=twy=tw(yvawy)=cva(wy)=c(vas)c$.
\end{center}
So $c$ is regular. From $yvawy=y$, we have $(yvaw-1)\in lann(y)\subseteq lann(yv)=lann(b)$. Further,  $yvawb=b$. Thus $b\R\subseteq yv\R$. The reverse inclusion $yv\R \subseteq b\R$ can be shown using Proposition \ref{prop14}(iv). Hence $yv\R=b\R$. For any $z\in rann(wy)$, we have $wyz=0$. Now $cz=twyz=0$. This  implies $z\in rann(c)$. Hence $rann(wy)\subseteq rann(c)$. On the other hand, if $x\in rann(c)$, then $cx=0$. Now $wyx=scx=0$. Therefore, $rann(wy)=rann(c)$.\\
(iv)$\Rightarrow$(v) It is enough to show $b$ is regular and $lann(yv)=lann(b)$. The regularity of $b$ and $lann(yv)=lann(b)$ can be proved in the similar way as (ii)$\Rightarrow$(iii).\\
(v)$\Rightarrow$(ii) Follows from Proposition \ref{prop14}(iii) and (iv).\\
(i)$\Rightarrow$(vi) Let $y=a^{v,w}_{b,c}$. Then there exist $s,t\in \R$ such that $y=bswy$ and $y=yvtc$. So $b=yvawb=bswyvawb=bswb$, implies $b$ is regular. Similarly, we can show $c$ is regular. Now
\begin{center}
  $bb^{-}y=bb^{-}bswy=bswy=y$ and $yvawbb^{-}=bb^{-}$.  
\end{center}
 Similarly we can show, $yc^{-}c=y$ and $c^{-}c=c^{-}cvawy$.\\
(vi)$\Rightarrow$(vii) If (v) holds, then $bb^{-}=yvawbb^{-}=yc^{-}cvawbb^{-}\in \R (c^{-}cvawbb^{-})$. Similarly we can show $c^{-}c\in (c^{-}cvawbb^{-})\R$.\\
(vii)$\Rightarrow$(viii) It is Obvious.\\
(viii)$\Rightarrow$(i) Let $bb^{-}=tc^{-}cvawbb^{-}$. Post-multiplying by $b$, we obtain $b=bb^{-}b=tc^{-}cvawb\in \R c^{-}cvawbb^{-}b=tc^{-}cvawb\in \R c^{-}cvawb$. Similarly, pre-multiplying $c$ to $c^{-}c=c^{-}cvawbb^{-}s$, we obtain 
$c=cvawbb^{-}s\in cvawb\R$. Hence by Lemma \ref{lem3.6}, we obtain  $a^{v,w}_{b,c}=y$.
\end{proof}

The relation between group inverse and $(v,w)$-weighted $(b,c)$-inverse is presented in the next result.

\begin{theorem}
Let $a,b,c,v,w\in \R$. Suppose $a^{v,w}_{b,c}$ exist and there exist $s\in \R$ such that $s\R=b\R$ and $rann(s)=rann(c)$. Then, $vaws,svaw\in \R^{\#}$ and $a^{v,w}_{b,c}=s(vaws)^{\#}=(svaw)^{\#}s$.
\end{theorem}

\begin{proof}
First, we will show that $vaws\in \R^{\#}$ and $a^{v,w}_{b,c}=s(vaws)^{\#}$. Let $g\in rann(vaws)$. Then  $vawsg=0$. Which implies $sg\in rann(vaw)\cap s\R=rann(vaw)\cap b\R=\{0\}$ by Theorem \ref{thm3.9}. It follows that $sg=0$ and $g\in rann(s)$. Thus, $rann(vaws)\subseteq rann(s)$ and consequently $rann(vaws)=rann(s)=rann(c)$. Since $s\R=b\R$, we have $vaws\R=vawb\R$. Using Theorem \ref{thm3.9}, we get 
\begin{center}
 $\R=vawb\R\oplus rann(c)=vaws\R\oplus rann(vaws)$.   
\end{center}
Thus, $1=vawsu+t$, for some $u\in \R$ and $t\in rann(vaws)$. Now $vaws=vawsvawsu$. This yields $vaws(vaws-vawsuvaws)=0$. Hence  $(vaws-vawsuvaws)\in rann(vaws)\cap vawsR=\{0\}$ and subsequently,
\begin{equation}\label{eq9}
      vaws=vawsuvaws=vawsvawsu.
  \end{equation}
  Clearly, $vawsu$ is idempotent. Using Proposition \ref{prop2.1}, we obtain  $rann(vawsu)=(1-vawsu)\R$. Using equation \eqref{eq9}, we obtain $rann(vawsu)\subseteq rann(vaws)$. Given any $h \in rann(vaws)$, then $vawsh=0$ for some $h\in \R$. So $vawsuh=vawsuvawsh=0$. Hence, $rann(vaws)\subseteq rann(vawsu)$ and thus $rann(vaws)=rann(vawsu)$. Again, $vaws-uvawsvaws\in rann(vaws)=rann(vawsu)$, hence we have $vawsu(vaws-uvawsvaws)=0$. From equation \eqref{eq9}, we get 
\begin{equation} \label{eq10}
vaws=vawsuvaws=(vawsu^2)vawsvaws.
\end{equation}
From equations \eqref{eq9} and \eqref{eq10}, we have $vaws\in \R(vaws)^2\cap (vaws)^2\R$. Hence by Lemma \ref{LemmaHar}, $vaws$ is group invertible. 

Next we will show that $s(vaws)^{\#}$ is the $(v,w)$-weighted $(b,c)$-inverse of $a$. Let $t=s(vaws)^{\#}$. Then 
\begin{center}
    $tvawt=s(vaws)^{\#}vaws(vaws)^{\#}=s(vaws)^{\#}=t$.
\end{center}
Clearly, $tv\R=(s(vaws)^{\#})v\R\subseteq sv\R\subseteq s\R=b\R$. Since $vaws((vaws)^{\#}vaws-1)=0$ and $rann(vaws)=rann(s)$, it follows that $s(vaws)^{\#}vaws=s$. Hence,
\begin{center}
   $bR=s\R=(s(vaws)^{\#}vaws)\R\subseteq tvaws\R\subseteq tv\R$. 
\end{center}
Similarly, we have 
\begin{eqnarray*}
    rann(wt)&=&rann(ws(vaws)^{\#})\subseteq rann(vaws(vaws)^{\#})=rann(vaws)=rann(s)\\
    &=&rann(c),
\end{eqnarray*}
and 
\begin{eqnarray*}
    rann(c)&=&rann(s)=rann(vaws)
    =rann(vaws(vaws)^{\#})\\
    &&\subseteq rann(s(vaws)^{\#}vaws(vaws)^{\#})
    =rann(s(vaws)^{\#})\\
    &=&rann(t)\subseteq rann(wt).
\end{eqnarray*}
Hence by Proposition \ref{pro3.7} and Theorem \ref{thm3.4}(iv), we obtain $a^{v,w}_{b,c}=t=s(vaws)^{\#}$.  Similarly, it can be shown that $svaw\in \R^{\#}$ and $a^{v,w}_{b,c}=(vaws)^{\#}s$.
\end{proof}

\begin{theorem} \label{thm4.11}
Let $a,v,w \in \R$. If $e,f\in \R$ with $e^2=e$ and $f^2=f$, then the following are equivalent;
\begin{enumerate}[\rm(i)]
 \item $e\in e\R fvawe$ and $f\in fvawe\R f$.
 \item there exist $m,n\in \R$ such that $p=mfvawe+1-e$ is invertible, and $fvawep^{-1}n=f$.
 \item there exist $m,n\in \R$ such that $q=fvawen+1-f$ is invertible and $mq^{-1}fvawe=e$;
 \item there exist $m,n\in \R$ such that $p=mfvawe+1-e$ and $q=fvawen+1-f$ are invertible.
\end{enumerate}
\end{theorem}

\begin{proof}
(i)$\Rightarrow$ (ii),(iii):  Let $e\in \R fvawe$ and $f\in fvawe\R$. Then there exist $m,n\in \R$ such that $e=mfvawe$ and $f=fvawen$. Take $p=mfvawe+1-e$ and $q=fvawen+1-f$. Then $p=q=1$, $f=fvawen=fvawep^{-1}n$ and $e=mq^{-1}fvawe$.\\
(ii)$\Rightarrow$(i) From $f=fvawep^{-1}n$ and $pe=mfvawe$, we have $e=p^{-1}mfvawe$. Post-multiplying $f=fvawep^{-1}n$ by $f$ and pre-multiplying $e=p^{-1}mfvawe$ by $e$, we obtain $e\in e\R fvawe$ and $f\in fvawe\R f$.\\
(iii)$\Rightarrow$ (i) Similar to (ii)$\Rightarrow$(i).\\
(ii)$\Rightarrow$(iv) Using (iii), we have $f=fvawep^{-1}n$.  Let $n_{1}=p^{-1}n$, then $fvawen_{1}+1-f=f+1-f=1$. So there exist $m,n_{1}$ such that $p=mfvawe+1-e$ and $q=fvawen_{1}+1-f$ are invertible.\\
(iv)$\Rightarrow$(i) Let $p=mfvawe+1-e$. Then  $pe=mfvawe$ and subsequently, $e=p^{-1}mfvawe$. Now $e=e^{2}=ep^{-1}mfvawe\in e\R fvawe$. Similarly, we can show $f\in fvawe\R f$.
\end{proof}

Following the Definition \ref{botdef}, we present following characterizations for $(v,w)$-weighted  Bott-Duffin $(e,f)$-inverse.% of $a \in \R$ 

\begin{proposition}\label{prop4.10}
Let $a,v,w,e,f\in \R$ with $e^2=e$ and $f^2=f$. If $a^{b,v,w}_{e,f}$ exists then $e\in e\R fvawe$ and $f\in fvawe\R f$.
\end{proposition}
\begin{proof}
Let $z=a^{b,v,w}_{e,f}$. Then by  Definition \ref{botdef}, 
\begin{equation*}
    \begin{split}
    e&=zvawe=ewzvawe=e(wzv)fvawe\in e\R fvawe, \textnormal{~~and~~}\\ 
    f&=fvawz=fvawzvf=fvawe(wzv)f\in fvawe\R f.\qedhere 
\end{split}
\end{equation*}
\end{proof}

\begin{theorem}
 Let $a,e,f,v,w\in \R$ such that $e=e^*=e^2$ and $f=f^*=f^2$. If $a^{b,v,w}_{e,f}$ exists, then the following are hold:
 \begin{enumerate}[(i)]
     \item $e\in \R (fvawe)^{*}fvawe$ and $f\in fvawe(fvawe)^{*}\R$;
     \item $p=(fvawe)^{*}fvawe+1-e$ is invertible and $fvawep^{-1}(fvawe)^{*}=f$.
     \item  $q=fvawe(fvawe)^{*}+1-f$ is invertible and $(fvawe)^{*}q^{-1}fvawe=e$
 \end{enumerate}
\end{theorem}

\begin{proof}
(i) Let $a^{b,v,w}_{e,f}$ exists. Then by Proposition \ref{prop4.10} and Theorem \ref{thm4.11}, we get $r=gfvawe+1-e$ is invertible and $fvawer^{-1}h=f$ for some $g,h\in \R$. Using this, we have $e=r^{-1}gfvawe$. Now 
\begin{eqnarray*}
e^{*}&=&(r^{-1}gfvawe)^{*}=(fvawe)^{*}(r^{-1}g)^{*}=(fvawe)^{*}f(r^{-1}g)^{*}\\
&=&(fvawe)^{*}fvawer^{-1}h(r^{-1}g)^{*}.
\end{eqnarray*}
Thus $e=(r^{-1}g)(r^{-1}h)^{*}(fvawe)^{*}fvawe\in \R (fvawe)^{*}fvawe$. Similarly, we can show that
$f=fvawe(fvawe)^{*}(r^{-1}g)^{*}r^{-1}h$ and $f\in fvawe(fvawe)^{*}\R $.\\
(ii) From part (i), we have $e=r^{-1}gfvawe$ and $fvawer^{-1}h=f$. So $r^{-1}gf=er^{-1}h$. Let $\beta=(r^{-1}g)(r^{-1}h)^{*}$. Then 
\begin{eqnarray*}
\beta e&=&\beta e^{*}=(r^{-1}g)(r^{-1}h)^{*}(fvawe)^{*}(r^{-1}g)^{*}=r^{-1}gf(r^{-1}g)^{*}\\
&=&er^{-1}h(r^{-1}g)^{*}=e\beta^{*}
\end{eqnarray*}
Thus $(\beta e+1-e)(fvawe)^{*}fvawe+1-e)=1=(fvawe)^{*}fvawe+1-e)(\beta e+1-e)$. Hence $p=(fvawe)^{*}fvawe+1-e$ is invertible and $p^{-1}=(\beta e+1-e)$. Further,
\begin{equation*}
    \begin{split}
    fvawep^{-1}(fvawe)^{*}&=fvawe(\beta e+1-e)(fvawe)^{*}=fvawe\beta (fvawe)^{*}\\
    &=fvawer^{-1}g(r^{-1}h)^{*}(fvawe)^{*}=fvawer^{-1}gf\\
    &=fvawer^{-1}h=f. \qedhere 
    \end{split}
\end{equation*}
(iii) Analogous to (ii).
\end{proof}

\section{Hybrid (v,w)-weighted (b,c)-inverse}
First we discuss an equivalent definition of hybrid  $(v, w)$-weighted  $(b, c)$-inverse, which will help us to prove more characterizations of this inverse.

\begin{theorem}\label{thm3.1}
 Let $a,b,c,v,w,y\in \R$ with either $v$ or $w$ be invertible. Then the followings are equivalent.
\begin{enumerate}[\rm(i)]
\item $yvawy=y, yv\R=b\R \mbox{ and  }rann(wy)=rann(c)$.
\item $yvawb=b,~ cvawy=c,~yv\R\subseteq b\R \mbox{ and }rann(c) \subseteq rann(wy)$.
\end{enumerate}
\end{theorem}
\begin{proof}
(i)$\Rightarrow$(ii) Let $yvawy=y$ and $yv\R=b\R$. Then there is  $t\in \R$ such that $b=yvt=yvawyvt=yvawb$. From $yvawy=y$, we obtain $1-vawy$ $\in rann(y)\subseteq rann(wy)=rann(c)$. Thus $c=cvawy$.\\
 (ii)$\Rightarrow$(i) Let $yv\R\subseteq b\R$. Then  $yv=br$ for some $r\in \R$. Multiplying $yvawb=b$ by $rv^{-1}$ on the right gives $yvawy=y$. If $w$ is invertible then $yvawy=y$ is similarly follows from $cvawy=c$. Using $yvawb=b$, we get $b\R\subseteq yvR$, and hence $yv\R=b\R$. Now, let $s\in rann(wy)$.  Then $wys=0$. Further, $s \in rann(c)$ since $cs=cvawys=0$. Hence $rann(wy)=rann(c)$.
\end{proof}

In view of Lemma \ref{Th4.7Drazin}, we explore a necessary condition for  hybrid $(v,w)$-weighted $(b,c)$-inverse in the below result.

\begin{theorem}\label{thm3.41}
Let $a,b,c,v,w,t\in\R$ with either $v$ or $w$ be invertible. If the hybrid $(v,w)$-weighted $(b,c)$-inverse of $a$ exists, then $bt$ is the hybrid $(v,w)$-weighted $(b,c)$-inverse of $a$ satisfying  $c=cvawbt$. 
\end{theorem}
\begin{proof}
Let $a$ has a hybrid $(v,w)$-weighted $(b,c)$-inverse. Then by Lemma \ref{Th4.7Drazin}, 
\begin{center}
   $c=cvawbt$ for some $t\in\R$. 
\end{center}
 Let $y=bt$. Now we will claim that $y$ is the hybrid $(v,w)$-weighted $(b,c)$-inverse of $a$. Clearly, $yv\R=btv\R \subseteq b\R$. 
 Using  $c=cvawbt$, we obtain $cvawb=cvawbtvawb$. Thus $(1-tvawb)\in rann(cvawb) \subset rann(b)$ by Lemma \ref{Th4.7Drazin}. Hence 
\begin{center}
  $b=btvawb=yvawb.$  
\end{center}
For $x\in rann(c)$, we have $cvawbtx=cx=0$. Which yields $tx\in rann(cvawb)$. Further, by Lemma \ref{Th4.7Drazin},  $tx\in rann(b)$ and consequently $wyx=wbtx=0$. Therefore, 
\begin{center}
    $rann(c) \subseteq rann(wy)$. 
\end{center}
By Theorem \ref{thm3.1}, we get $y=bt$ is the hybrid $(v,w)$-weighted $(b,c)$-inverse of $a$. 
\end{proof}

We now present a few necessary and sufficient condition for hybrid $(v,w)$-weighted $(b,c)$.

\begin{theorem}\label{thm3.5}
Let $a,b,c,v,w\in \R$ with either $v$ or $w$ be invertible. Then $a^{h,v,w}_{b,c}$ exists if and only if $\R=vawb\R\oplus rann(c)$ and $rann(vaw)\cap b\R=0$. 
\end{theorem}
\begin{proof}
Let $a$ has a hybrid $(v,w)$-weighted $(b,c)$-inverse. Then by Theorem  \ref{thm3.41}, $bt$ is the hybrid $(v,w)$-weighted $(b,c)$-inverse of $a$ satisfying $c=cvawbt$, where $t\in\R$. Subsequently $z:=(1-vawbt)\in rann(c)$.  For any $x\in\R$, we can write 
\begin{center}
  $x=1\cdot x=(z+vawbt)x=zx+vawbtx\in rann(c)+vawb\R$.   
\end{center}
Thus $\R=rann(c)+vawb\R$ since the reverse inclusion is trivial. If  $r\in rann(c)\cap vawb\R$, then $cr=0$ and $r=vawbu$ for some $u\in \R$. From Theorem \ref{thm3.1}, taking $y=bt$, we have 
\begin{equation}\label{eq1}
  rann(wbt)=rann(c) \mbox{ and } btvawb=b  
\end{equation}
 Which leads $wbtvawbu=wbtr=0$ and  $r=vaw(b)u=va(wbtvawbu)=0$. Hence $R=vawb\R\oplus rann(c)$. Next we will show that $rann(vaw)\cap b\R=0$. Let $h\in rann(vaw)\cap b\R$, then $vawh=0$ and $h=bk$ for some $k\in \R$. Which implies $vawbk=0$. Using second part of equation \eqref{eq1}, we have $h=(b)k=bt(vawbk)=0$. Thus $rann(vaw)\cap b\R=0$.\\
Conversely, let $\R=vawbR\oplus rann(c)$. Then $1=vawbm+n$ for some $m\in \R$ and $n\in rann(c)$. Further,
\begin{equation}\label{eq2}
    c=cvawbm+cn=cvawbm\in cvawb\R \mbox{ since }n\in rann(c).
\end{equation}
If $x\in rann(cvawb)$, then $cvawbx=0$ and hence $vawbx\in rann(c)\cap vawb\R=0$. Thus $bx\in rann(vaw)\oplus b\R=0$ since $rann(vaw)\cap b\R=0$.  Hence $x\in rann(b)$ and subsequently, we obtain 
\begin{equation}\label{eq3}
    rann(cvawb)\subseteq rann(b).
\end{equation}
In view of equations \eqref{eq2}, \eqref{eq3}, and Lemma \ref{Th4.7Drazin}, $a$ has a hybrid $(v,w)$-weighted $(b,c)$-inverse.
\end{proof}

 \begin{lemma}
Let $a,b,c,v,w\in \R$ with either $v$ or $w$ be invertible. Assume that $a^{h,v,w}_{b,c}$ exists. Then $a^{v,w}_{b,c}$ exists if and only if any one of the following is holds.
\begin{enumerate}[\rm(i)]
\item $cvawb$ is regular.
\item $c$ is regular.
\end{enumerate}
\end{lemma}
\begin{proof}
(i) Let $a^{v,w}_{b,c}$ exists. Then $b\in \R cvawb$ and $c\in cvawb \R$. Which implies $b=scvawb$ and $c=cvawbt$ for some $s,t\in \R$. In addition, 
\begin{center}
  $b=scvawb=scvawbtvawb=btvawb$, and 
  $c=cvawbt=cvawscvawbt=cvawsc$. 
\end{center}
Now 
\begin{eqnarray*}
cvawb&=&cvaw(btvawb)=cvawbtvawb=cvaw(btvawb)tvawb\\
&=&cvawbtvaw(bt)vawb=cvawbtvaw(sc)vawb\\
&=&cvawb(tvaws)cvawb.
\end{eqnarray*}
Hence $cvawb$ is regular.

Conversely, let $cvawb$ be regular. Then there exist an element $z\in \R$ such that $cvawb=cvawbzcvawb$. Since $a$ has a  hybrid $(v,w)$-weighted $(b,c)$-inverse, by Lemma \ref{Th4.7Drazin}, we have $c\in cvawb \R$ and  $b=bzcvawb\in \R cvawb$ due to the fact that $1-zcvawb\in rann(cvawb)\subseteq rann(b)$. Hence by Lemma \ref{lem3.6} $a$ has a $(v,w)$-weighted $(b,c)$-inverse.

.\\
(ii) Let $a$ has a $(v,w)$-weighted $(b,c)$-inverse. Then by Lemma \ref{lem3.6}, we have $c\in cvawb \R$ and $b\in \R cvawb$. Which yields  $c=cvawbh$ and $b=kcvawb$ for some $h,k\in \R$. Now 
\begin{center}
 $c=cvaw(b)h=cvaw(kcvawbh)=cvawkc=c(vawk)c$.   
\end{center}
Hence $c$ is regular.

Conversely, let $c$ be regular and the hybrid $(v,w)$-weighted $(b,c)$-inverse of $a$ be exists. Then by Theorem \ref{thm3.5}, $\R=vawb \R\oplus rann(c)$ and subsequently $1=vawbs+t$ where $s\in \R$ and $t\in rann(c)$. Therefore, $c=cvawbs+ct=cvawbs$. Since $c$ is regular, there exist an element $x\in \R$ such that $c=cxc$. Now 
\begin{center}
    $cvawb=(c)vawb=(c)xcvawb=cvawbsxcvawb=cvawb(sx)cvawb$.
\end{center}
Thus $cvawb$ is regular. Hence by part (i), $a$ has a $(v,w)$-weighted $(b,c)$-inverse.
\end{proof}

We next present the following characterizations of hybrid $(v,w)$-weighted $(b,c)$ through annihilators.

\begin{theorem}\label{thm3.10}
Let $a,b,c,v,w\in \R$ with either $v$ or $w$ be invertible. If $a$ has a hybrid $(v,w)$-weighted $(b,c)$-inverse, then the following statements are holds:
\begin{enumerate}[\rm(i)]
\item $rann(vawb)=rann(b)$. 
\item If $rann(b)=rann(c)$, then $rann(vawbs)=rann(vawb)$, where $s\in \R$ satisfies $vawb=(vawb)^2s$.
\end{enumerate}
\end{theorem}

\begin{proof}
(i) It is trivial that $rann(b)\subseteq rann(vawb)$. Let $a$ has a hybrid $(v,w)$-weighted $(b,c)$-inverse. Then by Theorem \ref{thm3.5}, $rann(vaw)\cap b \R=0$. For $r\in rann(vawb)$, we have $br\in rann(vaw)$ and $br\in b \R$. Thus $br\in rann(vaw)\cap b \R=0$ and hence $r\in rann(b)$. Therefore, $rann(vawb)\subseteq rann(b)$. 

(ii) Using the condition $R=vawb \R\oplus rann(c)$  of Theorem \ref{thm3.5}, we have $1=vawbs+t$ for some $s\in \R$ and $t\in rann(c)=rann(b)$. Thus $b=bvawbs$ and $vawb=vawbvawbs=(vawb)^2s$. Let $x\in rann(vawb)$. Then $(vawb)^2sx=vawbx=0$. Now 
\begin{center}
$vawbsx\in rann(vawb)\cap vawb \R=rann(b)\cap vawb \R=rann(c)\cap vawb \R=0$.   
\end{center}
Therefore $x\in rann(vawbs)$ and hence $rann(vawb)\subseteq rann(vawbs)$. 

Conversely, let $z\in rann(vawbs)$. Then $vawbz=(vawb)^2sz=0$. Which implies $z\in rann(vawb)$.
Thus, $rann(vawbs)\subseteq rann(vawb)$ and hence $rann(vawbs)= rann(vawb)$, where  $s$ satisfies $vawb=(vawb)^2s$.
\end{proof}

%We now discuss the following characterizations of hybrid $(v,w)$-weighted $(b,c)$ through annihilators

The following result represent a necessary and sufficient condition for hybrid $(v,w)$-weighted $(b,c)$ inverse through group inverse. 

\begin{theorem}\label{thm11}
Let $a,b,c,v,w\in \R$ with either $v$ or $w$ be invertible. Assume that $rann(vawb)=rann(b)=rann(c)$. Then the hybrid $(v,w)$-weighted $(b,c)$-inverse of $a$ exists if and only of $vawb$ is group invertible. 
\end{theorem}

\begin{proof}
Let $a$ has a hybrid $(v,w)$-weighted $(b,c)$-inverse. Then by Theorem \ref{thm3.5}, we have $1=vawbs+t$ for some $s\in \R$ and $t\in rann(c)=rann(vawb)$. Which implies  
\begin{equation}\label{eq4}
 vawb=(vawb)^2s\in (vawb)^2 \R \mbox{ and }   (vawb)^2=(vawb)^2svawb.
\end{equation}
 Using the second part of equation \eqref{eq4} and Theorem \ref{thm3.5}, we obtain
\begin{center}
  $vawb-vawbsvawb\in rann(vawb)\cap vawb \R =rann(c)\cap vawb \R=0$.  
\end{center}
Thus 
\begin{equation}\label{eq5}
  vawb=vawbsvawb \mbox{ and }(vawb)^2=vawbs(vawb)^2.
\end{equation}
Applying equation \eqref{eq5} and Theorem \ref{thm3.10}, we have 
\begin{center}
    $vawb-s(vawb)^2\in rann(vawb) = rann(vawbs)$.
\end{center}
Further, $vawbs^2(vawb)^2=vawbsvawb=vawb$ and $vawb\in \R (vawb)^2$. Hence by Lemma \ref{LemmaHar}, $vawb$ is group invertible since $vawb\in (vawb)^2\R\cap \R (vawb)^2$.

Conversely, let $y=b(vawb)^{\#}$. Since $vawb=(vawb)^2(vawb)^{\#}$ and $rann(vawb)=rann(c)$, it follow that $c(1-vawb(vawb)^{\#})=0$ and 
\begin{center}
   $c=cvawb(vawb)^{\#}=cvawy$.
\end{center}
Similarly, by applying $rann(vawb)=rann(b)$, we obtain 
\begin{center}
   $b=b(vawb)^{\#}vawb=yvawb$.
\end{center}
The condition $yv\R\subseteq y\R\subseteq bR$ follows from $y=b(vawb)^{\#}$. Next we will show that $rann(c)\subseteq rann(wy)$. Let $x\in rann(c)=rann(b)$. Then $bx=0$. Now 
$wyx=b(vawb)^{\#}x=b(vawb)^{\#}(vawb)^{\#}vawbx=0$. Thus $x\in rann(wy)$ and hence $rann(c)\subseteq rann(wy)$. By Theorem \ref{thm3.1}, $y=b(vawb)^{\#}$ is the hybrid $(v,w)$-weighted $(b,c)$-inverse of $a$.
\end{proof}
\begin{remark}\label{rem}
If $a,b,c,v,w\in \R$ with either $v$ or $w$ be invertible. If $vawb$ is group invertible, then $b(vawb)^{\#}$ is the hybrid $(v,w)$-weighted $(b,c)$-inverse of $a$.
\end{remark}

\begin{corollary}\label{cor13}
Let $a,b,c,v,w\in \R$ with either $v$ or $w$ be invertible. Assume that $a^{h,v,w}_{b,c}$ exists. Then $rann(b)=rann(vawb)=rann(c)$ if and only of $vawb$ is group invertible.
\end{corollary}

\begin{proof}
Let $vawb$ be group invertible and $y$ be the hybrid $(v,w)$-weighted $(b,c)$-inverse of $a$. Then  $rann(c)=rann(wy)$ and $y=b(vawb)^{\#}$ by Remark \ref{rem}. If $x\in rann(vawb)$, then by Theorem \ref{thm3.5}, $bx\in rann(vaw)\cap b \R=0$ and hence  $x\in rann(b)$. Thus $rann(vawb)=rann(b)$ since the reverse inclusion $rann(b)\subseteq rann(vawb)$ is obvious. Next we will show that $rann(b)=rann(wy)$. Let $z\in rann(b)$. Then $bz=0$ and consequently 
\begin{center}
 $wyz=wb(vawb)^{\#}z=wb(vawb)^{\#}(vawb)^{\#}vawbz=0$. 
\end{center}
Therefore,  $z\in rann(wy)$ and $rann(b)\subseteq rann(wy)$. If $x\in rann(wy)$, then $wyx=0$ and 
\begin{center}
    $vawbx=(vawb)^2(vawb)^{\#}x=vawbvawyx=0$.
\end{center}
Further, by Theorem \ref{thm3.5}, we obtain 
$bx\in rann(vaw)\cap b \R=0$. Thus $bx\in rann(b)$. Hence $rann(b)= rann(wy)=rann(c)$.

The converse part follows from Theorem \ref{thm11}.
\end{proof}

%The relation among the inverse of $a$ along $d$, the $(v,w)$-weighted $(d,d)$-inverse and the hybrid $(v,w)$-weighted $(d,d)$-inverse of $a$ by annihilators and ideals. Now, a lemma is given.

The following result present hybrid $(v,w)$-weighted inverse in the relationships with
 annihilators and $(v,w)$-weighted inverse of $a$ along $d \in \R$.

\begin{theorem}\label{thm3.15}
Let $a,d,v,w,y\in \R$ with either $v$ or $w$ invertible. Then the following statements are equivalent:
\begin{enumerate} [\rm(i)]
    \item $y$ is the  $(v,w)$-weighted inverse of $a$ along $d$.
    \item $yvawd=d=dvawy$, $\R wy\subseteq \R d$, and $lann(d)\subseteq lann(yv)$. 
    \item $yvawy=y$, $\R wy=\R d$, and $lann(yv)=lann(d)$.
    \item $yvawd=d=dvawy$, $yv \R\subseteq d \R$, and $rann(wy)=rann(d)$.
    \item $yvawy=y$, $yv \R=d \R$, and  $rann(wy)=rann(d)$.
    \item $y$ is the hybrid $(v,w)$-weighted $(d,d)$-inverse of $a$.
    \item $y$ is the $(v,w)$-weighted $(d,d)$-inverse of $a$.
\end{enumerate}    
\end{theorem}

\begin{proof}
(i)$\Rightarrow$(ii):  The proof is follows from the definition of $(v,w)$-weighted inverse of $a$ along $d$ and Proposition \ref{prop14} (ii).\\
(ii)$\Rightarrow$(iii): Let $\R wy \subseteq \R d$. Then $wy=sd$ for some $s\in\R$. Pre-multiplying $d=dvawy$ by $w^{-1}s$, we obtain
\begin{center}
    $y=w^{-1}sd=w^{-1}sdvawy=yvawy$.
\end{center}
From $d=dvawy$, we have $\R d\subseteq \R wy$ and hence $\R wy=\R d$. Next we will show that $lann(yv)\subseteq lann(d)$. If $z\in lann(yv)$, then by applying $yvawd=d$, we obtain
\begin{center}
 $zd=z(yvawd)=(zyv)awd=0$.   
\end{center} 
Thus $lann(yv)\subseteq lann(d)$ and consequently $lann(d)=lann(yv)$.\\
(iii)$\Rightarrow$(iv): Let $y=yvawy$. Then $(1-yvaw)\in lann(y)\subseteq lann(yv)=lann(d)$. Thus $d=yvawd$. From $\R d=\R wy$, we have  
\begin{equation}\label{eq6}
   d=swy \mbox{ and } wy=td\mbox{ for some }s,t\in \R.
\end{equation}
Pre-multylying $yvawy=y$ by $sw$, we obtain 
$d=sw(y)=(swy)vawy=dvawy$. The condition $rann(wy)=rann(d)$ follows from Proposition \ref{prop14} (i). Using the second part of equation \eqref{eq6}, we get $d$ is regular since 
\begin{center}
    $d=dvawy=d(vat)d$.
\end{center}
Hence by Proposition \ref{prop14} (iv), $yv \R\subseteq d \R$. \\
(iv)$\Rightarrow$(v): The proof is similar to (ii)$\Rightarrow$(iii).\\
(v)$\Leftrightarrow$(vi): This part is trivial and follows from the definition.\\
(vi)$\Rightarrow$(vii): Let $y$ be the hybrid $(v,w)$-weighted $(d,d)$-inverse of $a$. Then $yvawy=y$,  $yvR=dR$, and $rann(wy)=rann(d)$.  To establish the result, it is sufficient to show 
\begin{center}
  $yvawd=d=dvawy$ and $y\in yv\R d\cap d\R wy$.   
\end{center}
Let $y-yvawy$. Then $(1-vawy)\in rann(y)\subseteq rann(wy)=rann(d)$. Thus $d=dvawy$. From $yv \R=d \R$, we have $d=yvs$ and $yv=dt$ for some $s,t\in \R$. Therefore,
\begin{center}
$y=yvawy=d(ta)wy\in d \R wy$,   $d=yvs=yvawyvs=yvawd$ and $d=d(taw)d$.
\end{center}
Hence $d$ is regular and by Proposition \ref{prop14} (iii), we obtain $\R wy\subseteq \R d$. Which implies $wy=zd$ for some $z\in \R$ and  $y=yvawy=yv(az)d\in yv \R d$. Hence $y$ is the $(v,w)$-weighted $(d,d)$-inverse of $a$.\\
(vii)$\Rightarrow$(i): Let $y$ be the $(v,w)$-weighted $(d,d)$-inverse of $a$. Then $yvawd=d=dvawy$, and $y=dswy=yvtd$ for some $s,t\in \R$. To establish the result, it is enough to show $yv \R\subseteq d \R$ and $\R wy\subseteq \R d$. Since $yv=d(swyv)=ds_1$ and $wy=wyvtd=t_1d$ for some $s_1=swyv\in \R$ and $t_1=wyvt\in \R$, it follows that $yv\R\subseteq d\R$ and $\R wy\subseteq\R d$. Hence by Definition \ref{walongd}, $y$ is the $(v,w)$-weighted inverse of $a$ along $d$.
\end{proof}

In view of Theorem \ref{thm3.15}, and taking $b=c=d$ in Theorem \ref{thm3.9}, we obtain the following result as a corollary.

\begin{corollary}
Let $a,d,v,w\in \R$ with either $v$ or $w$ be invertible. Then the following statements are equivalent:
\begin{enumerate}[\rm(i)]
\item $a$ has a $(v,w)$-weighted inverse along $d$.
\item $d$ is regular, $\R=\R dvaw\oplus lann(d)$, and $lann(vaw)\cap \R d=0$.
\item $\R=\R dvaw\oplus lann(d)$, $lann(vaw)\cap \R d=0$ and $dvawd$ is regular.
\item $d$ is regular, $\R=vawd \R\oplus rann(d)$, and  $rann(vaw)\cap d \R=0$.
\item $\R=vawd \R\oplus rann(d)$, $rann(vaw)\cap d \R=0$ and $dvawd$ is regular.
\end{enumerate}
\end{corollary}

The relation between $(v,w)$-weighted inverse along $d \in \R$ and group inverse of an element is discussed in the next result.
\begin{corollary}
Let $a,d,v,w\in \R$. Then $a^{v,w}_{\parallel d}$ exists if and only if $vawd$ is group invertible and $rann(vawd)= rann(d)$.
\end{corollary}
\begin{proof}
Let $y$ be the $(v,w)$-weighted inverse of $a$ along along $d$. Then by Theorem \ref{thm3.15}, $y$ is the hybrid $(v,w)$-weighted $(d,d)$-inverse of $a$ and $yvawd=d$. From the condition $yvawd=d$, we have $\R d\subseteq\R vawd$. Taking $wy=d$ and $d=vawd$ in Proposition \ref{prop14} (i), we obtain $rann(vawd)\subseteq rann(d)$. Hence $rann(vawd)= rann(d)$ since the reverse inclusion $rann(d)\subseteq rann(vawd)$ is trivial. Replacing $b$ and $c$ by $d$ in Corollary \ref{cor13}, we get 
$vawd$ is group invertible. The converse part follows from Theorem \ref{thm11}
\end{proof}

\section{Annihilator (v,w)-weighted (b,c)-inverse}

This section is devoted to the characterizations  of annihilator (v,w)-weighted (b,c)-inverses.
The very first result is represent an equivalent definition of annihilator  $(v, w)$-weighted  $(b, c)$-inverse, which is used to discuss a few result on the  reverse  order  law  and properties of this inverse. 

\begin{theorem}\label{thm4.1}
Let $a,b,c,v,w,y\in \R$ with either $v$ or $w$ invertible. Then the following statements are equivalent: 
\begin{enumerate}[\rm(i)]
    \item $yvawy=y$, $rann(wy)=rann(c)$, and $lann(yv)=lann(b)$.
    \item $yvawb=b$, $cvawy=c$, $rann(c)\subseteq rann(wy)$, and $lann(b)\subseteq lann(yv)$. 
\end{enumerate}
\end{theorem}

\begin{proof}
(i)$\Rightarrow$(ii) Let $yvawy=y$. Then $(yvaw-1)\in lann(yv)=lann(b)$. This  yields $yvawb=b$. Similarly $cvawy=c$ follows from 
\begin{center}
    $(vawy-1)\in rann(wy)=rann(c)$.
\end{center}
Hence completes the proof since the range conditions are trivial.\\
(ii)$\Rightarrow$(i) Let $rann(c)\subseteq rann(wy)$ and $cvawy=c$. Then $(vawy-1)\in rann(c)\subseteq rann(wy)$. Thus  $wyvawy=wy$. Similarly from $lann(b)\subseteq lann(yv)$ and $yvawb=b$, we can obtain $yvawyv=yv$. Hence if either $v$ or $w$ is invertible then $yvawy=y$. Next we will claim that $rann(wy)\subseteq rann(c)$ and $lann(yv)\subseteq lann(b)$. For  $x\in rann(wy)$, we have $wyx=0$. Now $cx=cva(wyx)=0$. Thus $rann(wy)\subseteq rann(c)$. If $z\in lann(yv)$, then $zyv=0$. Further, $zb=(zyv)awb=0$. Hence  $lann(yv)\subseteq lann(b)$.
\end{proof}

With the help of Theorem \ref{thm4.1} (i), we present the following  property  of annihilator $(v,w)$-weighted $(b,c)$-inverse.

\begin{proposition}
For $i=1,2$, let $a_{i},b_{i},c_{i},v,w,y_{i}\in \R $ with both $v,w$ invertible and  $y_{i}={a_{i}}_{b,c}^{a,v,w}$. If $rc_{1}=c_{2}r$, $rva_{1}w=va_{2}wr$, and $rb_{1}=b_{2}r$  for any $r\in \R$, then $ry_{1}=y_{2}r$.
\end{proposition}

\begin{proof}
Let $y_{i}={a_{i}}_{b,c}^{a,v,w}$.  Then by Theorem \ref{thm4.1}, we obtain $y_{2}va_{2}wb_{2}=b_{2}$ and $lann(b_{1})\subseteq lann(y_{1}v)$. Thus 
\begin{center}
  $rb_{1}=b_{2}r$=$y_{2}va_{2}wb_{2}r$=$y_{2}(va_{2}wr)b_{1}$=$y_{2}(rva_{1}w)b_{1}$.
\end{center}
Therefore, $(r-y_{2}rva_{1}w)\in lann(b_{1})\subseteq lann(y_{1}v)$. Which implies 
\begin{equation}\label{eq7}
 ry_{1}v=y_{2}rva_{1}y_{1}v   
\end{equation}
Similarly, we have 
\begin{center}
 $c_{2}r=rc_{1}=rc_{1}va_{1}wy_{1}=c_{2}rva_{1}wy_{1}$, and  
\end{center} 
  $(r-rva_{1}wy_{1})\in rann(c_{2})\subseteq rann(wy_{2})$. Thus 
    \begin{equation}\label{eq8}
    wy_{2}r=wy_{2}rva_{1}y_{1}.  
  \end{equation}
Using the invetibility of $v$ and $w$ in equation \eqref{eq7} and \eqref{eq8}, we get $ry_{1}=y_{2}r$. 
\end{proof}

In the similar manner, we have the following result for the $(v,w)$-weighted $(b,c)$-inverse.
\begin{corollary}
For $i=1,2$, let $a_{i},b_{i},c_{i},v,w,y_{i}\in \R $ and  $y_{i}$ be the $(v,w)$-weighted $(b_{i},c_{i})$-inverse of $a_{i}$. If $rc_{1}=c_{2}r$, $rva_{1}w=va_{2}wr$, and $rb_{1}=b_{2}r$  for any $r\in \R$, then $ry_{1}=y_{2}r$.
\end{corollary}
\begin{proof}
We first note that, $(rva_{1}w)b_{1}=(va_{2}wr)b_{1}=va_{2}w(rb_{1})$. Similarly $rva_{1}wb_{1}=va_{2}wb_{2}r$ and $rc_{1}va_{1}w=c_{2}va_{2}wr$.

Since $y_{i}$ is the $(v,w)$-weighted $(b_{i},c_{i})$-inverse of $a_{i}$, we have $c_{1}va_{1}wy_{1}=c_{1}$ and $y_{2}a_{2}b_{2}=b_{2}$. Also we may write $y_{1}=b_{1}ewy_{1}$ and $y_{2}=y_{2}vfc_{2}$ for some $e,f\in \R$.

Now we find
\begin{center}
   $ry_{1}=r(b_{1}ewy_{1})=(rb_{1})ewy_{1}=(b_{2}r)ewy_{1}=(y_{2}va_{2}wb_{2})rewy_{1}=y_{2}(va_{2}wb_{2}r)ewy_{1}=y_{2}(rva_{1}wb_{1})ewy_{1}=y_{2}rva_{1}w(b_{1}ewy_{1})=y_{2}rva_{1}wy_{1}$.
\end{center}
and similarly 
\begin{center}
    $y_{2}r=(y_{2}vfc_{2})r=y_{2}vf(c_{2}r)=y_{2}vf(rc_{1})=y_{2}vfr(c_{1}va_{1}wy_{1})=y_{2}vf(rc_{1}va_{1}w)y_{1}=y_{2}vf(c_{2}va_{2}wr)y_{1}=(y_{2}vfc_{2})va_{2}wry_{1}=y_{2}(va_{2}wr)y_{1}=y_{2}(rva_{1}w)y_{1}$.

\end{center}
Hence  $ry_{1}=y_{2}r$.
\end{proof}

The next result concerning on the reverse order law for the annihilator $(v,w)$-weighted $(b,c)$-inverse.

\begin{theorem}
Let $s,t,b,c,v,w\in \R$ with $v$ or $w$ invertible and both $r^{a,v,w}_{b,c}$, $s^{a,v,w}_{b,c}$ exist. If $bvtw=vtwb$, and $cvsw=vswc$, then $(st)^{a,v,w}_{b,c}=t^{a,v,w}_{b,c}s^{a,v,w}_{b,c}$.
\end{theorem}

\begin{proof}
Let $y=t^{a,v,w}_{b,c}r^{a,v,w}_{b,c}$. Then we have 
\begin{center}
    $yvswvtwb=t^{a,v,w}_{b,c}s^{a,v,w}_{b,c}vswvtwb$=$t^{a,v,w}_{b,c}s^{a,v,w}_{b,c}vswbvtw$=$t^{a,v,w}_{b,c}bvtw$=$t^{a,v,w}_{b,c}vtwb=b$.
\end{center}
  Similalry, we can show $cvswvtwy$=$cvswvtwt^{a,v,w}_{b,c}s^{a,v,w}_{b,c}=c$. From the definition \ref{defann}, we have $lann(b)\subseteq lann(t^{a,v,w}_{b,c}v)$ and 
    $rann(c)\subseteq rann(wt^{a,v,w}_{b,c})$. Now for any $z\in lann(b)$, we obtain $zt^{a,v,w}_{b,c}v=0$ and 
    \begin{center}
     $zyv=zt^{a,v,w}_{b,c}s^{a,v,w}_{b,c}v=zt^{a,v,w}_{b,c}vtwt^{a,v,w}_{b,c}s^{a,v,w}_{b,c}v=0$. 
  \end{center}
   Hence $lann(b)\subseteq lann(yv)$. Let $z\in rann(c)\subseteq rann(wt^{a,v,w}_{b,c})$. Then $wt^{a,v,w}_{b,c}z=0$. Now
   \begin{center}
     $wyz=wt^{a,v,w}_{b,c}s^{a,v,w}_{b,c}z$=$wt^{a,v,w}_{b,c}s^{a,v,w}_{b,c}vswt^{a,v,w}_{b,c}z=0$. 
   \end{center}
   Thus $rann(c)\subseteq rann(wy)$. Hence by Theorem \ref{thm4.1} (ii), we obtain $(st)^{a,v,w}_{b,c}=y=t^{a,v,w}_{b,c}s^{a,v,w}_{b,c}$. 
\end{proof}

\begin{corollary}
Let $s,t,b,c,v,w\in \R$, and both  $s^{v,w}_{b,c}$,  $t^{v,w}_{b,c}$ exists. If $cvsw=vswc$ and $bvtw=vtwb$, then $(st)^{v,w}_{b,c} =t^{v,w}_{b,c}s^{v,w}_{b,c}$.
\end{corollary}

\begin{proof}
Let $y=t^{v,w}_{b,c}s^{v,w}_{b,c}$. Then 
\begin{center}
    $yvswvtwb=t^{v,w}_{b,c}s^{v,w}_{b,c}vswvtwb$=$t^{v,w}_{b,c}s^{v,w}_{b,c}vswbvtw$=$t^{v,w}_{b,c}bvtw$=$t^{v,w}_{b,c}vtwb=b$.
\end{center}
Similalry,  
\begin{center}
    $cvswvtwy$=$cvswvtwt^{v,w}_{b,c}s^{v,w}_{b,c}=vswcvtwt^{v,w}_{b,c}s^{v,w}_{b,c}=vswcs^{v,w}_{b,c}=cvsws^{v,w}_{b,c}=c$.
\end{center}

Since $s^{v,w}_{b,c}\in b\R ws^{v,w}_{b,c}$ and $t^{v,w}_{b,c}\in b\R wt^{v,w}_{b,c}$, we have  \begin{center}
    $y=t^{v,w}_{b,c}s^{v,w}_{b,c}\in b\R wt^{v,w}_{b,c}s^{v,w}_{b,c}= b\R wy$.

\end{center}
Also $s^{v,w}_{b,c}\in s^{v,w}_{b,c}v\R c$ and $t^{v,w}_{b,c}\in t^{v,w}_{b,c}v\R c$ gives 
\begin{center}
    $y=t^{v,w}_{b,c}s^{v,w}_{b,c}\in t^{v,w}_{b,c}s^{v,w}_{b,c}v\R c=yv\R c$.
\end{center}
Hence $(st)^{v,w}_{b,c} =t^{v,w}_{b,c}s^{v,w}_{b,c}$.
\end{proof}

The following result present annihilator  $(v,w)$-weighted $(b,c)$ inverse in the relationships with hybrid $(v,w)$-weighted $(b,c)$ inverse and  $(v,w)$-weighted inverse of $a$ along $d \in \R$.

\begin{proposition}
Let $a,v,w,y,e\in \R$ with either $v$ or $w$ invertible and $e$ is regular. Then the following conditions are equivalent:
\begin{enumerate}[\rm(i)]
    \item  $y=a^{v,w}_{\parallel e}$.
    \item  $yvawe=e=evawy$, $\R wy\subseteq \R e$, and $lann(e)\subseteq lann(yv)$. 
    \item  $yvawy=y$, $\R wy\subseteq \R e$, and $lann(yv)=lann(e)$.
    \item  $yvawe=e=evawy$, $yv\R\subseteq e \R$, and $rann(e)\subseteq rann(wy)$.
    \item  $yvawy=y$, $yv\R\subseteq e \R$, and $rann(wy)=rann(e)$.
    \item  $y=a^{h,v,w}_{e,e}$.
    \item  $y=a^{v,w}_{e,e}$. 
    \item  $y=a^{a,v,w}_{e,e}$.
\end{enumerate}
\end{proposition}

\begin{proof}
The equivalence of (i)$\Leftrightarrow$(vii) follows from Theorem \ref{thm3.15}. Next we will show that (vii)$\Leftrightarrow$(viii). Since $a^{v,w}_{e,e}$ is a special case of $a^{a,v,w}_{e,e}$, it is enough show (viii)$\Rightarrow$(vii). Let $y=a^{a,v,w}_{e,e}$. Then 
\begin{center}
     $yvawy=y$, $rann(wy)=rann(e)$, and  $lann(yv)=lann(e)$.
\end{center}
Clearly both $e$ and $y$ are regular. So by Proposition \ref{prop14}, $lann(yv)=lann(e)$ gives $yv\R \subseteq e\R$ and $rann(wy)=rann(e)$ gives $\R wy=\R e$. Hence $y=a^{v,w}_{e,e}$.
\end{proof}

\begin{lemma}
For $i=1,2$, let $a_{i},b_{i},c_{i},v,w,y_{i}\in \R $ with $v,w$ both invertible and  $y_{i}={a_{i}}_{b_i,c_i}^{a,v,w}$. If $b_{1}=b_{2}$ then $y_{1}va_{1}wy_{2}=y_{2}$ and $y_{2}va_{2}wy_{1}=y_{1}$. Mutually If $c_{1}=c_{2}$ then $y_{1}va_{2}wy_{2}=y_{1}$ and $y_{2}va_{1}wy_{1}=y_{2}$.
\end{lemma}

\begin{proof}
Let $y_{i}={a_{i}}_{b_i,c_i}^{a,v,w}$. Then by Theorem \ref{thm4.1},  $y_{1}va_{1}wb_{1}=b_{1}$ and consequently,
\begin{center}
  $(y_{1}va_{1}w-1)\in lann(b_{1})=lann(b_{2})\subseteq lann(y_{2}v)$.  
\end{center}
 Thus $y_{1}va_{1}wy_{2}v=y_{2}v$. Post-multiplying by $v^{-1}$ we get $y_{1}va_{1}wy_{2}=y_{2}$.  From $y_{2}va_{2}wb_{2}=b_{2}$, we have 
 \begin{center}
   $(y_{2}va_{2}w-1)\in lann(b_{2})=lann(b_{1})\subseteq lann(y_{1}v)$.   
 \end{center}
 Therefore, $y_{2}va_{2}wy_{1}v=y_{1}v$. Again post-multiplying by $v^{-1}$, we obtain $y_{2}va_{2}wy_{1}=y_{1}$. In the similar manner, we can show that if $c_{1}=c_{2}$, then $y_{1}va_{2}wy_{2}=y_{1}$ and $y_{2}va_{1}wy_{1}=y_{2}$.
\end{proof}

\begin{theorem}
Let $a_{1},a_{2},v,w,b,c\in \R $ with both $v$ and $w$ invertible. Suppose $y_{1}={a_{1}}_{b,c}^{a,v,w}$ and $y_{2}={a_{2}}_{b,c}^{a,v,w}$. Then $y_{1}+y_{2}=y_{1}v(a_{1}+a_{2})wy_{2}=y_{2}v(a_{1}+a_{2})wy_{1}$.
\end{theorem}

\begin{proof}
$y_{1}={a_{1}}_{b,c}^{a,v,w}$ and $y_{2}={a_{2}}_{b,c}^{a,v,w}$. Then by Theorem \ref{thm4.1}, we have $y_{1}va_{1}wb=b$ and $y_{2}va_{2}wb=b$. Thus $(y_{1}va_{1}w-1)=lann(b)\subseteq lann(y_{2}v)$ yields $y_{1}va_{1}wy_{2}v=y_{2}v$. Post-multiplying by $v^{-1}$, we get $y_{1}va_{1}wy_{2}=y_{2}$. Similarly we can show  $y_{2}va_{2}wy_{1}=y_{1}$. Again using Theorem  \ref{thm4.1}, we have $cva_{1}wy_{1}=c$. Which implies 
\begin{center}
    $(va_{1}wy_{1}-1)\in rann(c)\subseteq rann(wy_{2})$.
\end{center}
Thus $y_{2}va_{1}wy_{1}=wy_{2}$. Pre-multiplying by $w^{-1}$, we get  $y_{2}va_{1}wy_{1}=y_{2}$. Similarly, we can show $y_{1}va_{2}wy_{2}=y_{1}$. Now
\begin{center}
   $y_{1}v(a_{1}+a_{2})wy_{2}=y_{1}va_{1}wy_{2}+y_{1}va_{2}wy_{2}=y_{2}+y_{1}$, and 
\end{center}
\begin{center}
   $y_{2}v(a_{1}+a_{2})wy_{1}=y_{2}va_{1}wy_{1}+y_{2}va_{2}wy_{1}=y_{2}+y_{1}$.  
\end{center}
\end{proof}

\section{Conclusion}
We have discussed a few necessary and sufficient conditions for the existence of the $(v, w)$-weighted $(b,c)$ inverses of elements in rings. Derived representations are used in generating corresponding representations of the  $(v,w)$-weighted hybrid $(b,c)$-inverse and annihilator $(v, w)$-weighted $(b, c)$-inverses. We have also explored a few new  results related to the reverse order law for  annihilator $(v, w)$-weighted $(b, c)$-inverses. In addition to this, the notion of $(v,w)$-weighted Bott-Duffin $(e,f)$-inverse introduced and a few characterisations of this inverse are studied.\\

\noindent{\bf{Acknowledgments}}\\
Ratikanta Behera is grateful to the Mohapatra Family Foundation and the College of Graduate Studies, University of Central Florida, Orlando, FL, USA, for their financial support for this research.

\section*{Disclosure statement}

No potential conflict of interest was reported by the authors.

\section*{Funding}

The Mohapatra Family Foundation and the College of Graduate Studies, University of Central Florida, Orlando, FL, USA,

\begin{comment}
\section*{ORCID}

Jajati Keshari Sahoo\orcidA \href{http://orcid.org/0000-0001-6104-5171}{http://orcid.org/0000-0001-6104-5171}
\end{comment}
%\medskip
\bibliographystyle{abbrv}
\bibliography{reference}
\end{document}